\documentclass[12pt]{amsart}
\usepackage{hyperref}
\usepackage{physics}

\newtheorem{theorem}{Theorem}

\newtheorem{proposition}[theorem]{Proposition}

\newtheorem{question}[theorem]{Question}
\newtheorem{lemma}[theorem]{Lemma}

\theoremstyle{definition}
\newtheorem{example}[theorem]{Example}

\begin{document}

\begin{abstract}Gilmer \cite{Gilmer} has recently shown that in any nonempty union-closed family $\mathcal F$ of subsets of a finite set, there exists an element contained in at least a proportion $.01$ of the sets of $\mathcal F$. We improve the proportion from $.01$ to $\frac{ 3 -\sqrt{5}}{2} \approx .382$ in this result. An improvement to $\frac{1}{2}$ would be the Frankl union-closed set conjecture. We follow Gilmer's method, replacing one key estimate by a sharp estimate. We then suggest a new addition to this method and sketch a proof that it can obtain a constant strictly greater than $\frac{ 3 -\sqrt{5}}{2} $. We also disprove a conjecture of Gilmer that would have implied the union-closed set conjecture. \end{abstract}

\author{Will Sawin}

\title{ An improved lower bound for the union-closed sets conjecture} 

\maketitle

We prove the following weak form of the Frankl union-closed conjecture.

\begin{theorem}\label{uc-progress} Let $\mathcal F$ be a nonempty union-closed family of subsets of $[n]$. Then there exists $i\in [n]$ contained in a proportion at least $\frac{3- \sqrt{5}}{2}$ of the sets in $\mathcal F$. \end{theorem}

Theorem \ref{uc-progress} follows from Theorem \ref{entropy-sharp}, which itself follows from Lemma \ref{inductive-sharp}.  For a random variable $A$ valued in sets, we let $H(A)$ be the entropy of $A$. For $ p \in [0,1]$, we let $H(p) =-p \log p - (1-p) \log (1-p)$ be the Shannon entropy, i.e. the entropy of a Bernoulli random variable with parameter $p$. (We never use $H$ to denote the entropy of a real-valued random variable, so there will be no ambiguity.)

\begin{theorem}\label{entropy-sharp} Let $u \in [0,1]$. Let $A$ and $B$ denote independent samples from a distribution over subsets of $[n]$. Assume that, for all $i\in [n]$, $ \operatorname{Pr} [ i \in A] \leq u$. Then

\[ H(A \cup B) \geq H(A) \cdot  \begin{cases} \frac{ H ( 2u -u^2)}{H(u)} & \textrm{ if } u\leq \frac{3 -\sqrt{5}}{2} \\  (1-u) \frac{2}{\sqrt{5}-1 } & \textrm { if } u \geq  \frac{3 -\sqrt{5}}{2} \end{cases}.\]

\end{theorem}

\begin{lemma}\label{inductive-sharp} Let $u \in [0,1]$. Let  $p, q$ be  i.i.d. $[0,1]$-valued random variables with expectation $\leq u$. Then 
\[ \mathbb E [  H( p + q -  p q ) ] \geq   \mathbb E[ H(p) ] \cdot \begin{cases} \frac{ H ( 2u -u^2)}{H(u)} & \textrm{ if } u\leq \frac{3 -\sqrt{5}}{2} \\  (1-u) \frac{2}{\sqrt{5}-1 } & \textrm { if } u \geq  \frac{3 -\sqrt{5}}{2} \end{cases}   \]\end{lemma}

All three of these represent quantitative improvements of corresponding results of Gilmer \cite[Theorem 1, Theorem 2, and Lemma 1]{Gilmer}, who in particular proved a lower bound of $.01$ for the maximum proportion of the sets in a union-closed family $\mathcal F$ containing an element. The work of Gilmer itself improved on work of Knill \cite{Knill} and W\'ojick \cite{Wojick} who proved lower bounds on the proportion comparable to $\frac{1}{ \log \abs{\mathcal F}}$.

 The deduction of Theorem \ref{uc-progress} from Theorem \ref{entropy-sharp}, and Theorem \ref{entropy-sharp} from Lemma \ref{inductive-sharp}, is identical to the one by Gilmer, and relies on taking $A$ and $B$ independent samples from the uniform distribution on $\mathcal F$ and estimating the entropy by induction on the restrictions of $A$ and $B$ to $[i]$ for $i$ from $1$ to $n$. The innovation is entirely in the proof of Lemma \ref{inductive-sharp}.

 Furthermore, Lemma \ref{inductive-sharp} and Theorem \ref{entropy-sharp} are completely sharp -- there are examples meeting the inequality for any particular value of $u\in (0,1)$. Indeed:

\begin{example}\label{sharp-1} Theorem \ref{entropy-sharp} is sharp for $u\leq  \frac{3-\sqrt{5}}{2}$ because of the example, due to Gilmer, where the events $i\in A$ for $i \in [n]$ are independent of probability $u$. In this case, $A$ has entropy $n H(u)$ and the events $i\in A \cup B$ are independent of probability $2u-u^2$ so $A \cup B$ has entropy $H( 2u-u^2)$.  \end{example}

\begin{example}\label{sharp-2} Theorem \ref{entropy-sharp} is sharp for  $u \geq  \frac{3-\sqrt{5}}{2}$ because of the example where, with probability $ (1-u) \frac{2}{\sqrt{5}-1} $, we choose each element $i\in [n]$ to lie in $A$ independently with probability $\frac{ 3- \sqrt{5}}{2}$ and with probability $ 1- (1-u) \frac{2}{\sqrt{5}-1} $, we choose  $A=[n]$. Then $A$ has entropy $ (1-u) \frac{2}{\sqrt{5}-1} n  H(\frac{3-\sqrt{5}}{2})+O(1)$ because a convex combination of two probability distributions has entropy that differs from the convex combination of their entropies by $\leq \log 2$. Furthermore, with probability $ \left((1-u) \frac{2}{\sqrt{5}-1} \right)^2$,  each element $i\in [n]$ lies in $A\cup B$ independently with probability $\frac{\sqrt{5}-1}{2} $ and with probability $ 1- \left( (1-u) \frac{2}{\sqrt{5}-1} \right)^2$, we have $A \cup B=[n]$. Thus $A \cup B$ has entropy  $ \left( (1-u) \frac{2}{\sqrt{5}-1}\right)^2 n  H(\frac{\sqrt{5}-1}{2})+O(1)$ for the same reason. Dividing, and ignoring the lower-order $O(1)$ term, we see that Theorem \ref{inductive-sharp} is sharp since $H(\frac{3-\sqrt{5}}{2}) =H(\frac{\sqrt{5}-1}{2})$. This gives a negative answer to a question of Gilmer \cite[first bulleted question on p.9]{Gilmer} for $u >  \frac{3-\sqrt{5}}{2}$, while Theorem \ref{entropy-sharp} gives a positive answer to that question for $u \leq   \frac{3-\sqrt{5}}{2}$. \end{example}
 
 The sharpness of Lemma \ref{inductive-sharp} arises from, for $u \leq  \frac{3-\sqrt{5}}{2}$, a random variable equal to $u$ with probability $1$, or, if  $u \geq  \frac{3-\sqrt{5}}{2}$, a random variable equal to $ \frac{3-\sqrt{5}}{2}$ with probability $(1-u) \frac{2}{\sqrt{5}-1} $ and equal to $1$ with probability $1- (1-u) \frac{2}{\sqrt{5}-1} $. 
 
 Furthermore, we demonstrate that, contrary to \cite[Conjecture 1 on p. 9]{Gilmer}, incorporating the KL divergence does not improve the estimate:
 
\begin{proposition}\label{counterexample} For any $u <  1$ and $ d > \frac{ H ( 2u -u^2)}{H(u)} $, for all sufficiently large $n$, there exist a random variable $A$ valued in subsets of $[n]$, containing each element with probability $\leq u$, such that $H(A \cup B)  \leq d H(A)$, but $D ( A \cup B|| A) =O(1)$ while $H(A)$ grows linearly in $n$. \end{proposition}

The proof of Lemma \ref{inductive-sharp} uses calculus of variations to replace the quadratic form $ \mathbb E [  H( p + q -  p q ) ]$ with a bilinear form. The key idea is that, while we cannot use Jensen's inequality since the relevant functions like $H( p+q-pq) - \lambda H(p)$ are not convex, we can show that a relevant function is convex on one region and concave on another region, and there are also strong (but very different) inequalities for the expectation of a concave function.

Finally, we sketch a proof of Theorem \ref{uc-progress} with the lower bound replaced by $\frac{3- \sqrt{5}}{2}+\delta$ for some $\delta>0$. The idea is to consider, in addition to $A$ and $B$ independent uniform samples for $\mathcal F$, also $A$ and $C$ uniform but correlated samples from $\mathcal F$. We choose the correlation greedily to maximize the entropy increase at each step in the inductive argument. This gives a gain in entropy compared to the independent samples but also a loss as the correlations in previous steps can cause problems in future steps. However, for distributions close to the optimum in Lemma \ref{inductive-sharp}, the gain is greater than the loss, so we gain overall by considering a suitable linear combination of the two entropies.
 
 The author was supported by NSF grant DMS-2101491 while working on the paper and would like to thank Ryan Alweiss, Stijn Cambie, Zachary Chase, and Lucas Gerardo Hern\'andez Ch\'avez for helpful comments on earlier versions of this manuscript. I would especially like to thank Bhavik Mehta for pointing out an error in the original proof of Lemma \ref{increase-decrease} and Ravi Bopanna for pointing me to the techniques for a corrected proof based on his proof in \cite{Bopanna} of a special case sufficient for Theorem \ref{uc-progress}.
 
 The same day this article first appeared on arXiv, two independent proofs \cite{AHS,CS} of Theorem \ref{uc-progress} also did. Since \cite{CS} depended on \cite{AHS} for a key lemma, and the proof of a similar lemma in the first version of this article contained an error, only \cite{AHS} gave a completely independent proof. The next day, an independent construction \cite{Ellis} of a counterexample similar to Proposition \ref{counterexample} (for $n=2$ instead of for fixed $n$) appeared.

\section{Proofs}

\begin{proof}[Proof of Theorem \ref{uc-progress} using Theorem \ref{entropy-sharp} ] The proof is identical to \cite[p. 2]{Gilmer}, but we repeat it here for convenience.

Let $A$ and $B$ be independent uniform samples from $\mathcal F$.  Assuming for contradiction that there does not exist $i\in [n]$ contained in a proportion at least $\frac{3- \sqrt{5}}{2}$ of the sets in $\mathcal F$, there must be $u < \frac{3 -\sqrt{5}}{2}$ for which each element of $[n]$ is contained in a proportion $\leq u$ of sets in $\mathcal F$ and thus for which $ \operatorname{Pr} [ i \in A] \leq u$ for all $i \in n$. We then have
\[ H(A \cup B) \geq H(A) \cdot  \frac{H ( 2u -u^2)}{H(u)}  > H(A) \] which contradicts the fact that the uniform distribution on $\mathcal F$ is the maximum-entropy random variable supported on $\mathcal F$. \end{proof}

Let $\lambda =  \begin{cases} \frac{ H ( 2u -u^2)}{H(u)} & \textrm{ if } u\leq \frac{3 -\sqrt{5}}{2} \\  (1-u) \frac{2}{\sqrt{5}-1 } & \textrm { if } u \geq  \frac{3 -\sqrt{5}}{2} \end{cases} $.


\begin{proof} [Proof of Theorem \ref{entropy-sharp} using Lemma \ref{inductive-sharp}] The proof is identical to \cite[Proof of Theorem 1 on p. 4]{Gilmer}, but we repeat it here for convenience.

 Let  $A_{<i}$ be the intersection of $A$ with $[i-1]$, and similarly for $B_{<i}$ and $(A \cup B)_{<i} = A_{<i} \cup B_{<i}$. We prove
 \begin{equation}\label{inductive-inequality} H((A \cup B)_{<i} ) \geq H(A_{<i}) \cdot  \lambda \end{equation}
 by induction on $i$. The case $i= n+1$ will prove the theorem. The base case $i=1$ is trivial as both sides are $0$.  For the induction step, by the chain rule for entropy, we have \[ H((A \cup B)_{_{<i+1} } )  =  H (  (A \cup B)_{< (i+1) }  | (A \cup B)_{<i} ) +  H ( (A \cup B)_{<i} ) \]  and \[ H(A _{_{<i+1} } )  =  H (  A _{< (i+1) }  | A _{<i} ) +  H ( A_{<i} ) \] so assuming \eqref{inductive-inequality} for $i$, to obtain \eqref{inductive-inequality} for $i+1$ it suffices to check
 \begin{equation}\label{induction-step} H   ((A \cup B)_{< (i+1) }  | (A \cup B)_{<i} ) \geq \lambda H   (A _{< (i+1) }  | A _{<i} ).\end{equation}
 
 We have \[H (  (A \cup B)_{< (i+1) }  | (A \cup B)_{<i} )   \geq H ( (A \cup B)_{<(i+1)} | A_{<i}, B_{<i})\] by the data processing inequality. Let $p_i = \operatorname{Pr} [i \in A | A_{<i}]$ and $q_i = \operatorname{Pr} [i \in B| B_{<i}]$ be the conditional probabilities. Then because $A$ and $B$ are independent and identically distributed, $p_i$ and $q_i$ are independent, identically distributed random variables with expectation $\leq u$.  Furthermore, the probability that $i \in (A \cup B)_{< (i+1) } $ conditional on $A_{<i}$ and $B_{<i}$ is $p_i + q_i -p_i q_i$.  Since $i \in A$ is the only information contained in $A_{<(i+1)}$ but not $A_{<i}$, we have $H (  A _{< (i+1) }  | A _{<i} ) = \mathbb E[H (p_i)] $ and similarly $H ( (A \cup B)_{<(i+1)} | A_{<i}, B_{<i}) =  \mathbb E[ H( p_i+q_i- p_i q_i)]$. \eqref{induction-step} then follows from Lemma \ref{inductive-sharp}.

 \end{proof}

\begin{proof} [Proof of Lemma \ref{inductive-sharp}] 
An equivalent statement is that the minimum value of 
\begin{equation}\label{first-minimum}  \mathbb E_{(p,q) \sim \mu \times \mu}  [  H( p + q -  pq ) ] - \lambda \mathbb E_{p \sim \mu} [ H(p)] \end{equation} among all probability measures $\mu$ on $[0,1]$ with expectation $\leq u$ is nonnegative.

The minimum value is attained by a measure since every sequence of probability measures $\mu$ on $[0,1]$ has a weakly convergent subsequence and the functions $H(p+q-pq)$, $H(p)$, and $p$ are all continuous, so their expectations over weakly convergent sequences of measures converge.

Let us first check that, if $\mu$ an optimal measure, then $\mu$ also minimizes
\begin{equation}\label{second-minimum} 2 \mathbb E_{(p,q) \sim \mu \times \nu}  [  H( p + q -  p q ) ] - \lambda \mathbb E_{q \sim \nu} [ H(q)] = \mathbb E_{q\sim \nu} \left[ 2  \mathbb E_{p\sim \mu}   [  H( p + q -  p q ) ]  - \lambda H( q) \right] \end{equation}  among all probability measures $\nu$ on $[0,1]$ with expectation $\leq u$. Indeed, otherwise, taking $\mu' = (1-\epsilon) \mu + \epsilon \nu$
\[  \mathbb E_{(p,q) \sim \mu' \times \mu'}  [  H( p + q -  p q ) ] - \lambda \mathbb E_{p \sim \mu'} [ H(p) ]\] 
\[= (1-\epsilon)^2 \mathbb E_{(p,q) \sim \mu \times \mu}  [  H( p + q -  p q ) ]  + 2 \epsilon (1-\epsilon) \mathbb E_{(p,q) \sim \mu \times \nu}  [  H( p + q -  p q ) ] \] \[ + \epsilon^2\mathbb E_{(p,q)\sim \nu \times \nu}  [  H( p + q -  p q ) ] - \lambda (1-\epsilon) \mathbb E_{p\sim \mu} [ H(p) ] - \lambda \epsilon\mathbb E_{p \sim \mu} [ H(p) ] \] 
\[ = \mathbb E_{(p,q)\sim \mu \times \mu}  [  H( p + q -  p q ) ] - \lambda \mathbb E_{p\sim \mu }[ H(p)]  \] \[ + \epsilon \left(2 \mathbb E_{(p,q)\sim \mu \times \nu}  [  H( p + q -  p q ) ] - \lambda \mathbb E_{p\sim \nu }  [ H(p)]   -  2\mathbb E_{(p,q)\sim \mu \times \mu}  [  H( p + q -  p q ) ] +  \lambda \mathbb E_{p\sim \mu} [ H(p)] \right) + O(\epsilon^2) \]
which if the coefficient of $\epsilon$ is negative will be $<  \mathbb E_{(p,q)\sim \mu \times \mu}  [  H( p + q -  p q ) ] - \lambda \mathbb E_{p\sim \mu} [ H(p)] $ for small $\epsilon$, contradicting the minimality of $\mu$, and verifying that the minimal value of \eqref{second-minimum} is attained by $\nu=\mu$.

We now study the function $ F_\mu(q)=  2  \mathbb E_{p \sim \mu }  [  H( p + q -  pq  )   - \lambda H( q)]$. Specifically, we calculate $  \frac{d}{dq} q (1-q) \frac{d^2}{ d q^2}  F_\mu(q) $.

To do this, note that $H(q) = - q \log q - (1-q) \log (1-q)$, so $\frac{d}{dq} H(q) =  -\log q  + \log (1-q) $ and $\frac{d^2}{dq^2} H(q) =- \frac{1}{q}  - \frac{1}{1-q} =- \frac{1}{q(1-q)}$.  By the chain rule for a linear change of coordinates, $\frac{d^2}{dq^2} H(p+q-pq) = - \frac{ (1-p)^2}{ (p+q-pq) ( 1-p-q +pq) } =-\frac{ 1-p}{ (p+q-pq) (1-q)} $. Thus
\[ \frac{d^2}{ d q^2}  F_\mu(q)  = - 2 \mathbb E_{p \sim \mu } \left[ \frac{ 1-p}{ (p+q-pq) (1-q)} \right] +  \lambda \frac{1}{q(1-q)},\]
\[ q (1-q) \frac{d^2}{ d q^2}  F_\mu(q)  =  - 2 \mathbb E_{p \sim \mu } \left[ \frac{ (1-p)q}{ p+q-pq} \right] +  \lambda, \]
and therefore
\[ \frac{d}{dq} q (1-q) \frac{d^2}{ d q^2}  F_\mu(q)  = - 2\mathbb  E_{p \sim \mu } \left[ \frac{(1-p)  p }{ (p+q-pq)^2 } \right] <0 ,\] at least outside the degenerate case when $\mu$ is supported on $0$ and $1$ and \eqref{first-minimum} is zero.

It follows that $ q(1-q) \frac{d^2}{ d q^2}  F_\mu(q) $ is strictly decreasing. In particular, it either takes positive values on some interval $[0, a)$, zero value at $a$, and negative values on $(a,1]$, or is positive on all of $[0,1]$, or negative on all of $[0,1]$.

$ \frac{d^2}{ d q^2}  F_\mu(q) $ is positive or negative exactly where  $ q(1-q) \frac{d^2}{ d q^2}  F_\mu(q) $ is. Thus $F_\mu(q)$ is either strictly convex on some interval $[0,a)$ and strictly concave on $(a,1]$, convex on the whole interval, or concave on the whole interval.

If $\mu$ minimizes the expectation of $F_{\mu}(q)$ then $\mu$ must assign zero measure to the concave interval $(a,1)$, except its boundary points $\{a,1\}$, since otherwise we could push the mass to the boundary while preserving the expectation of $q$ and lower the expectation of $F_\mu(q)$. Similarly, $\mu$ restricted to the convex interval $[0,a]$ must be an atom since otherwise we could push the mass to the center  while preserving the expectation of $q$ and lower the expectation of $F_\mu(q)$. 

So $\mu$ is atomic, supported on the point $1$ and at most one other point. (In the convex or concave case, the reasoning is similar but simpler). It therefore suffices to show \eqref{first-minimum} is nonnegative for all measures $\mu$ of this form.

Let $\mu$ place mass $w$ on the point $v$ and mass $1-w$ on the point $1$. Then $\mathbb E_{p \sim \mu } [ H(p)] $ is  $w H(v)$ and $ \mathbb E_{(p,q) \sim \mu \times \mu}  [  H( p + q -  pq ) ] $ is $w^2 H( 2v-v^2)$, and thus \[  \mathbb E_{(p,q)\sim \mu \times \mu}  [  H( p + q-  pq ) ] - \lambda \mathbb E_{p \sim \mu } [ H(p)] \geq 0\] as long as \[ w^2 H(2v-v^2) \geq \lambda w H(v),\] i.e. as long as  \[ w \frac{ H(2v-v^2) }{H(v) } \geq \lambda.\]  Since the left side is increasing in $w$, this holds for all such measures with expectation $1- w (1-v) \leq u$ if and only if it holds for $w = \frac{1-u}{1-v}$, i.e. if and only if
\[  (1-u)  \min_{ v \in [0,u] } \left( \frac{ H(2v-v^2)}{ H(v) (1-v)} \right)  \geq \lambda .\]

By Lemma \ref{increase-decrease} below, $\frac{ H(2v-v^2)}{ H(v) (1-v)} $ is decreasing for $v < \frac{ 3-\sqrt{5}}{2} $ and increasing for $v > \frac{3 -\sqrt{5}}{2}$. Furthermore at $v= \frac{3 -\sqrt{5}}{2}$, we have  $\frac{ H(2v-v^2)}{ H(v)}=1 $ so $\frac{ H(2v-v^2)}{ H(v) (1-v)}= \frac{2}{\sqrt{5}-1}$. It follows that \[\min_{ v \in [0,u] } \left( \frac{ H(2v-v^2)}{ H(v) (1-v)} \right)   =  \begin{cases} \frac{ H ( 2u -u^2)}{H(u)(1-u) } & \textrm{ if } u\leq \frac{3 -\sqrt{5}}{2} \\  \frac{2}{\sqrt{5}-1 }& \textrm { if } u \geq  \frac{3 -\sqrt{5}}{2} \end{cases}   = \frac{\lambda}{(1-u)} , \] 
giving the claim. \end{proof}

\begin{lemma}\label{easier-entropy-inequality} For all $s\in (0,1)$, we have $H(s^2)< 2s H(s) $. \end{lemma}

\begin{proof}  We have $2 (1-s) > (1-s^2)$ since their difference is $(1-s)^2>0$ and we have \[-s \log (1-s) = s^2 + s^3/2 + s^4/3 + \dots > s^2 + s^4/2 + s^6/3+\dots =- \log (1-s^2) \] so 
\[ H(s^2) = -s^2 \log (s^2) - (1-s^2)\log(1-s^2) < -s^2 \log(s^2) - 2 (1-s) \log(1-s^2) \] \[ < -s^2 \log(s^2) -  2s(1-s) \log(1-s)  = 2 s H(s) .\]\end{proof}

\begin{lemma}\label{increase-decrease} The function $\frac{ H(2v-v^2)}{ H(v) (1-v)} $ from $(0,1)$ to $\mathbb R$ is decreasing for $v < \frac{ 3-\sqrt{5}}{2} $ and increasing for $v > \frac{3 -\sqrt{5}}{2}$. \end{lemma}

\begin{proof} 
Making the change of variables $s =1-v$, the expression $\frac{ H(2v-v^2)}{ H(v) (1-v)}$ simplifies slightly to $ \frac{ H(s^2)}{ s H(s) }$, which we denote by $F(s)$. We must show $F(s)$ is decreasing for $s< \frac{\sqrt{5}-1}{2} = \phi^{-1} $ and increasing for $s> \phi^{-1}$. Owing to the fact that  $\phi^{-1} $ is a local minimum of  $F(s)$ (by direct calculation or \cite{Bopanna}), it suffices to show the derivative of $F(s)$ is nonvanishing for all $s \in (0,1)  \setminus \{\frac{\sqrt{5}-1}{2}\}$.

Fix an $s_0$ where the derivative of $F(s)$ vanishes, from which we will derive a contradiction, and let $\beta = F(s_0)$. Then $\frac{ H(s^2)}{ s H(s) } - \beta $ vanishes to second order at $s_0$ so $ H(s^2) - \beta s H(s)$ vanishes to second order at $s_0$ as well.

The function $ H(s^2) - \beta s H(s)$ also vanishes to second order at $0$, and to first order at $1$. If those are the only zeroes of $H(s^2) - \beta s H(s)$ in the interval $[0,1]$, then $H(s^2) - \beta s H(s)$ never changes sign on $(0,1)\setminus \{s_0\}$, making $s_0$ either a unique global minimum or a unique global maximum of $F(s)$. But this is impossible as we know $\phi^{-1}$ is a global minimum \cite[Lemma on p. 2]{Bopanna} and the global maximum is not attained since $F(s)<2$ for all $s\in (0,1)$ by Lemma \ref{easier-entropy-inequality}  but $F(s)$ converges to $ 2$ as $s \to 0$ or $s\to 1$ by an easy calculation. 

So $H(s^2) -\beta s H(s)$ must have at least one more zero and thus it has zeroes in $[0,1]$ of total order at least six, so its third derivative has zeroes in $[0,1]$ of total order at least three by Rolle's theorem.

 Differentiating $H(s^2) -  \beta s H(s)$ three times, we obtain
\[ \frac{d^3}{ds^3} H(s^2)  = \frac{d^2}{ ds^2} 2s H'(s^2) = \frac{d}{ds} ( 2  H'(s^2) + 4s^2 H''(s^2) ) = 12 s H''(s^2)+ 8s^3 H'''(s^2) 
\] \[=  \frac{ -12 s  }{ s^2 (1-s^2) }  +  \frac{ 8 s^3 ( 1- 2s^2) }{  (s^2 (1-s^2))^2} =  \frac{ -4  -4 s^2 } {  s (1-s^2)^2} \]
and
\[ \frac{d^3}{ds^3} s H(s) =3 \frac{d^2}{ds^2} H(s) +s \frac{d^3}{ds^3} H(s) =\frac{-3}{s(1-s)}   + s \frac{ 1 - 2s }{ s^2 (1-s)^2}= \frac{ s (1-2s) - 3 s (1-s)} { s^2 (1-s)^2} = \frac{ s-2}{ s(1-s)^2}  .\]
so
\[ \frac{d^3}{ds^3} H(s^2) -\beta s H(s)= \frac{  -4-4s^2 - \beta (s-2) ( 1+s)^2 }{ s  (1-s^2)^2} .\]

The numerator is a polynomial in $s$ of degree $3$ with leading coefficient $-\beta$. Since the leading coefficient is negative, and the numerator takes the value $ -4 + 2 \beta<0$ at $0$ (because $\beta<2$ by Lemma \ref{easier-entropy-inequality}), the numerator must have at least one negative real zero. Thus it has at most two zeroes in $[0,1]$, giving the desired contradiction.

\end{proof}

\begin{proof}[Proof of Proposition \ref{counterexample}] Fix $\overline{u}<u$ such that $d>  \frac{ H ( 2\overline{u} -\overline{u}^2)}{H(\overline{u} )}  $. 

Define a distribution $A$ as follows. First generate a nonnegative-integer-valued random variable $k$ according to the geometric distribution with parameter $\theta$, i.e. the probability of attaining $k$ is $(1-\theta) \theta^k$. Then choose each $i\in [n]$ to lie in $A$ independently, uniformly, with probability $ 1 - (1-\overline{u})^{k+1}$. We choose $\theta$ sufficiently small to ensure various inequalities are satisfied.

Each element lies in $A$ with probability $ \sum_{k=0}^{\infty} (1-\theta)  \theta^k (1 - (1-\overline{u})^{k+1})$. As $\theta\to 0$, this converges to $\overline{u}$, so for $\theta$ sufficiently small this probability is at most $u$.

The entropy of $A$ is at least the average over values of $k$ of the entropy conditional on $k$, which is $\sum_{k=0}^{\infty} (1-\theta) \theta^k  n H(  (1-\overline{u})^{k+1})$. 

For $A, B$ two independent samples from the same distribution, let $k_A$ and $k_B$ be the $k$ values used to sample $A$ and $B$, and let $k' = k_A + k_B+1$. Then $k'$ is a positive integer, the probability of obtaining a given value $k'$ is $(1-\theta)^2 k' \theta^{k'-1}$, and each element $i\in [n]$ lies in $A \cup B$ independently, uniformly, with probability  $ 1 - (1-\overline{u})^{k'+1}$.

The entropy of $A \cup B$ is at most the entropy of the random variable $k'$, which is finite, plus the average over $k'$ of the entropy conditionally on $k'$, which is $\sum_{k'=1}^{\infty} (1-\theta) {k'} \theta^{k'-1}   n H(  (1-\overline{u})^{k'+1})$.  So the ratio of entropies is at most
\[  \frac{O(1) + \sum_{k'=0}^{\infty} (1-\theta) {k'} \theta^{k'-1}   n H(  (1-\overline{u})^{k'+1})}{ \sum_{k=0}^{\infty} (1-\theta) \theta^k  n H(  (1-\overline{u})^{k+1})} = o(1) +   \frac{ \sum_{k'=0}^{\infty} (1-\theta) {k'} \theta^{k'-1}    H(  (1-\overline{u})^{k'+1})}{ \sum_{k=0}^{\infty} (1-\theta) \theta^k  H(  (1-\overline{u})^{k+1})} .\]

As $\theta \to 0$, the numerator and denominator of the fraction converge to  $H ( 2 \overline{u}- \overline{u}^2)$ and $H(\overline{u})$ respectively, so for $\theta$ sufficiently small the fraction is strictly less than $d$, and then for $n$ sufficiently large the ratio of entropies is at most $d$.

Finally, by the convexity of KL divergence, \[D( A \cup B|| A)\leq \sum_{k'=1}^{\infty} (1-\theta) {k'} \theta^{k'-1} D( S_{k'}||A) ,,\] where each $i$ lies in $S_{k'}$  independently, uniformly, with probability  $ 1 - (1-\overline{u})^{k'+1}$.  The distribution of $S_{k'}$  can be obtained by conditioning $A$ on an event with probability $(1-\theta) \theta^{k'}$, so the probability that $S_{k'}=S$ is always at most $ \frac{1}{(1-\theta) \theta^{k'}}$ times the probability that $A=S$, giving \[D(S_{k'}||A) \leq  \log \left( \frac{1}{(1-\theta) \theta^{k'}}\right)  = - k' \log \theta  -\log (1-\theta)\] so that \[ D(A \cup B|| A) \leq \sum_{k'=1}^{\infty} (1-\theta) {k'} \theta^{k'-1} ( - k' \log \theta  -\log (1-\theta) ) = O(1) \] since the sum of a quadratic function against an exponentially decreasing function is bounded.

\end{proof}

\section{Proof Sketch}

We sketch a proof that there exists $\delta>0$ such that,  for $\mathcal F$ a nonempty union-closed family of subsets of $[n]$, there exists $i\in [n]$ contained in a proportion at least $\frac{3- \sqrt{5}}{2}+ \delta$ of the sets in $\mathcal F$. 

To do this, in addition to considering $A, B$ two independent uniform samples from $\mathcal F$, we choose $C$ a uniform sample from $\mathcal F$ that is not necessarily independent from $A$. Since $\mathcal F$ is union-closed, we  have $A \cup B, A \cup C \in \mathcal F$ and thus \[ H(A \cup B) , H(A \cup C) \leq \log \abs{\mathcal F} = H(A). \] We will prove that if each $i \in [n]$ is contained in $A$ with probability  $< \frac{3- \sqrt{5}}{2}+ \delta$  that \[ (1-\alpha) H(A \cup B) + \alpha H(A \cup C) > H(A).\] This will give a contradiction and thus let us conclude some $i$ is contained in $A$ with probability $\geq \frac{3- \sqrt{5}}{2}+ \delta$.

Let $A_i = 1$ if $i \in A$ and $0$ if $i\neq A$, and similarly for $B_i$ and $C_i$.

We describe a random process that, at the $i$th step, determines $A_i$ and $C_i$. Thus at the $i$th step $A_{<i}$ and $C_{<i}$ are fixed. We will choose this in such a way that $A$ and $C$ are uniformly distributed on $\mathcal F$.

Let $p_i$ be the proportion of $\{S \in \mathcal F \mid S_{<i} = C_{<i}\}$ that contain $i$ and let $r_i$ be the proportion of $\{S \in \mathcal F \mid S_{<i} = A_{<i}\}$ that contain $i$. We will choose $A_i$ to be a random variable that is $1$ with probability $p_i$ and $0$ with probability $1-p_i$, and choose $C_i$ to be $1$ with probability $r_i$ and $1$ with probability $1-r_i$. We will choose $A_i$ and $C_i$ to be correlated in a way that maximizes the conditional entropy of $\max (A_i, C_i)$. Specifically, if $p_i\geq 1/2$ or $r_i \geq 1/2$ we generate a uniformly random $x\in [0,1]$ and take $A_i =[ x<p_i]$ and $C_i =[x<r_i]$ so that \[\operatorname{Pr} [ \max(A_i,C_i)=1 |A_{<i}, C_{<i} ]= \max(p_i,r_i).\] and if $p_i, r_i<1/2$ we generate a uniformly random $x \in [0,1]$ and take $A_i = [x<p_i]$ and $C_i = [0 \leq 1/2- x< r_i] $, so that  \[\operatorname{Pr} [ \max(A_i,C_i)=1 |A_{<i}, C_{<i} ]= \min(p_i+r_i,1/2).\]

Since the conditional probability that $i\in A$ is $p_i$, the probability of getting any given sequence is the product of conditional probabilities which matches the probability when $A$ is uniformly distributed in $\mathcal F$. This shows $A$ and $C$ are uniformly distributed in $\mathcal F$.  Letting $q_i = \operatorname{Pr} [ i \in B | B_{<i}]$, we have
\[  (1-\alpha) H((A \cup B)_{<i+1} ) + \alpha H((A \cup C)_{<i+1} ) \] \[= (1-\alpha) H((A \cup B)_{<i} ) + \alpha H((A \cup C)_{<i} )+ (1-\alpha) H((A \cup B)_{<i+1} |  (A \cup B)_{<i}   ) + \alpha H((A \cup C)_{<i+1}  | (A \cup C) _{<i} )\]
\[  \geq (1-\alpha) H((A \cup B)_{<i} ) + \alpha H((A \cup C)_{<i} )+ (1-\alpha) H((A \cup B)_{<i+1} |  A_{<i}, B_{<i}   ) + \alpha H((A \cup C)_{<i+1}  | A_{<i}, C_{<i}  )\]
\[  \geq (1-\alpha) H((A \cup B)_{<i} ) + \alpha H((A \cup C)_{<i} )+ (1-\alpha)  \mathbb E [ H(p_i + q_i - p_i q_i ) ] + \alpha E [  H(\max( p_i, r_i , \min (p_i + r_i,1/2) ) )] .\]

Thus to inductively prove the entropy bound, it suffices to prove that for $p,q,r$ identically distributed $[0,1]$-valued random variables with expectation $\leq \frac{3- \sqrt{5}}{2}+ \delta$ with $p$ and $q$ independent but $p$ and $r$ not necessarily independent, we have
\begin{equation}\label{improved-inequality} (1-\alpha)  \mathbb E [ H(p+ q - p q ) ] + \alpha \mathbb E [  H(\max( p, r , \min (p + r,1/2) ) )] >  \mathbb E[ H(p) ] .\end{equation} 

We can do this by choosing $\alpha$ sufficiently small and $\delta$ sufficiently small depending on $\alpha$.  

The proof of Lemma \ref{inductive-sharp} shows that the only measures $\mu$ on $[0,1]$ with  $ \mathbb E_{(p,q) \sim \mu \times \mu}  [  H( p + q -  pq ) ] \leq  \mathbb E_{p \sim \mu} [ H(p)]$ and expectation $\leq \frac{3- \sqrt{5}}{2}$ are the delta measure at the point $ \frac{3- \sqrt{5}}{2}$ and measures supported on $\{0,1\}$.  So any weakly convergent sequence of measures with ratio $ \mathbb E_{(p,q) \sim \mu \times \mu}  [  H( p + q -  pq ) ] /  \mathbb E_{p \sim \mu} [ H(p)]$ converging to something $\leq 1$ and expectation convergent to something $\leq \frac{3- \sqrt{5}}{2}$ must converge to one of those two. It can't converge to a $\{0,1\}$-supported measure as measures close to that one with low expectation are easily seen to have high entropy ratio, so it must converge to the delta measure at $\{ \frac{3- \sqrt{5}}{2}\}$. It follows that for $\alpha, \delta$ sufficiently small depending on $\epsilon$, any measure with $\mathbb E_{(p,q) \sim \mu \times \mu}  [  H( p + q -  pq ) ] /  \mathbb E_{p \sim \mu} [ H(p)] \leq 1/(1-\alpha)$ and expectation $\leq \mathbb E_{(p,q) \sim \mu \times \mu}  [  H( p + q -  pq ) ] /  \mathbb E_{p \sim \mu} [ H(p)] + \delta$ must be close to the delta measure at $ \frac{3- \sqrt{5}}{2}$ in the sense that $\mathbb E [ \abs{ p - \frac{3- \sqrt{5}}{2} } ] <\epsilon$.

If $p$ and $r$ are both supported on such a measure, then both $p$ and $r$ are usually close to  $\frac{3- \sqrt{5}}{2}$, so, regardless of how $p$ and $r$ are correlated, $\max( p, r , \min (p + r,1/2) ) $ is usually close to $1/2$, and thus $E [  H(\max( p, r , \min (p + r,1/2) ) )] /  \mathbb E[ H(p) ]$ is close to $H(1/2) / H ( \frac{3- \sqrt{5}}{2} ) > 1$.   We can choose $\delta$ close enough to $0$, depending on $\alpha$, that $\mathbb E[ H(p+q-pq)] / \mathbb E[H(p)]$ is sufficiently close to $1$ to ensure that \eqref{improved-inequality} is satisfied.

It would be interesting to modify this argument to obtain an explicit value of $\delta$.

Motivated by this argument, we raise the question:

\begin{question} For any probability measure $\mu$ on subsets of $[n]$, with nonzero entropy, such that $\mu ( \{ A \subseteq [n] \mid i\in A\})<1/2$ for all $i\in [n]$, do there exist random variables $A, B$, identically distributed with measure $\mu$ but not necessarily independent, such that $H( A \cup B) > H(A) $? \end{question}

A positive answer would imply the union-closed conjecture.

\end{document}